\documentclass[10pt]{siamltex}
\usepackage{amsmath,amsfonts, amssymb,comment}
\usepackage{arydshln}
\usepackage{setspace}
\usepackage{graphicx}
\usepackage{pst-all}
 \usepackage{multirow}

\title{A generalization of the divide and conquer algorithm for the symmetric tridiagonal eigenproblem}
\author{
 Do Young Kwak\thanks{Korea Advanced Institute of Science and Technology,
Daejeon, Korea 305-701. {{\tt email:kdy@kaist.ac.kr}}. This work was supported by NRF, No.2014R1A2A1A11053889.}
\and Jaeyeon Kim \thanks{ Korea Advanced Institute of Science and Technology, Daejeon, Korea 305-701  {{\tt
email:jaeyeon@kaist.ac.kr }}.}} \vspace{1.06in}

\date{}

\def\eps{\epsilon}

\def\b0{{\bf 0}}
\def\be{{\bf e}}

\def\bu{{\mathbf u}}
\def\bv{{\mathbf v}}
\def\bw{{\mathbf w}}
\def\bx{{\mathbf x}}

\def\bz{{\mathbf z}}

\begin{document}
  \maketitle \pagestyle{myheadings} \thispagestyle{plain}
\markboth{Do Young Kwak and Jaeyeon Kim} {A generalization of the Cuppen's algorithm}

\begin{abstract}
In this paper, we present a generalized Cuppen's divide-and-conquer algorithm for the symmetric tridiagonal eigenproblem.
We extend the Cuppen's work to the rank two modifications of the form $A =T +\beta_1\bw_1\bw_1^T + \beta_2\bw_2\bw_2^T$, where $T$ is a block tridiagonal matrix having three blocks.  
We introduce a new deflation technique and obtain a  secular equation, for which the distribution of eigenvalues
 is nontrivial. We present a  way to count the number of eigenvalues in each subinterval. It turns out that each 
 subinterval contains either none, one or two eigenvalues. Furthermore, computing eigenvectors preserving the orthogonality are also suggested.
  Some numerical results, showing our algorithm can calculate the eigenvalue twice as fast as the Cuppen's divide-and-conquer algorithm, are included. 
\end{abstract}

\noindent {\bf Key words.} Eigenvalues of symmetric tridiagonal matrix, Cuppen's divide and conquer method,
Rank two modifications, Secular equation.

\noindent {\bf AMS subject classifications.} 65F15, 15A18

\section{Introduction}
Given a $n\times n$ tridiagonal real symmetric matrix $A$, the symmetric tridiagonal
eigenproblem is to find all the eigenvalues and the corresponding eigenvectors of $A$. In this paper we generalize the Cuppen's divide and conquer algorithm (CDC) \cite{Cuppen}  for
the symmetric tridiagonal eigenproblem by considering a rank two modification. The CDC  algorithm starts with decomposing $A$ as a rank one modification
\begin{equation}\label{rank-one-mod}
A =\left[
  \begin{array}{c;{2pt/2pt}c}
  T_1 &  \b0 \\ \hdashline[2pt/2pt]
 \b0  & T_2   \\
  \end{array}
  \right] +\beta \bw\bw^T,\quad \bw=\be_k+\be_{k+1},
\end{equation} where  $T_1, T_2$ are   symmetric   tridiagonal matrices of order
$k\geqq1$, $n-k\geqq1$ respectively, and $\beta\ne 0$ is $(k,k+1)$ entries of $A$.
By diagonalizing $T_i$ as $T_i=Q_i D_iQ_i^T$ and deflating the cases when the eigenvalues of $A$ coincide with the diagonals of $D_i$,
secular equation of CDC is derived. By using the eigenvalues obtained from the secular equation, we can
compute eigenvectors.

It is tempting to  generalize the  CDC algorithm to the following rank two modifications:
\begin{equation}
 \label{rank-two-modification2}
A =T +\beta_1\bw_1\bw_1^T + \beta_2\bw_2\bw_2^T,
  \end{equation}
 where $T$ is  block tridiagonal, and $\bw_1=\be_{k_1} + \be_{k_1+1}$,  $\bw_2=\be_{k_2}+\be_{k_2+1}$.
 Then  we mimic the CDC algorithm to compute eigenvalues and eigenvectors of (\ref{rank-two-modification2}) by using spectral decomposition of $T_i$'s.
We deflate the cases when some of the eigenvalues of $A$ coincide with the diagonals of $D_i$.
 Then we can derive the secular equation. By computing the eigenvalues from the secular equation, we can finally obtain the eigenvectors of $A$.

 \section{Dividing Step} \label{dividing step}

 We rewrite $A$ blockwise as   follows:
\begin{equation}\label{rank-two-modification}
\begin{array}{l l}
A
&=\left[
  \begin{array}{c;{2pt/2pt}c;{2pt/2pt}c}
  T_1 & \b0 & \b0\\ \hdashline[2pt/2pt]
 \b0  & T_2 & \b0\\ \hdashline[2pt/2pt]
\b0   &\b0 & T_3  \\
  \end{array}
  \right] +\beta_1 \left[
  \begin{array}{c;{2pt/2pt}c;{2pt/2pt}c}
  \ \ {}_1 &{}_1 \ \  &\b0 \\ \hdashline[2pt/2pt]
   \ \ {}^1& ^1\ \ &\b0\\ \hdashline[2pt/2pt]
\b0   &\b0 &\b0  \\
  \end{array}
  \right]  + \beta_2\left[
  \begin{array}{c;{2pt/2pt}c;{2pt/2pt}c}
\b0  &\b0 & \b0 \\ \hdashline[2pt/2pt]
  \b0 &\ \ {}_1 & {}_1 \ \ \\ \hdashline[2pt/2pt]
   \b0&  \ \ {}^1& ^1\ \ \\
  \end{array}
  \right]
  \end{array},
  \end{equation}
  where $T_1$, $T_2$ and $T_3$ are symmetric  tridiagonal matrices of order $k_1\geqq 1$, $k_2\geqq 1$, $k_3\geqq 1$ respectively, with  $k_1+k_2+k_3=n$ and $\beta_1\ne 0$, $\beta_2\ne 0$ are $(k_1,k_1+1)$, $(k_2,k_2+1)$ entries of
  $A$ respectively.
The computation of the spectral decomposition of each $T_i$
  $$T_i=Q_iD_iQ_i^T,\quad i=1,2,3$$
    is cheaper than (about $2/3$) those of CDC  where  $D_i$ are diagonal matrices and  $Q_i^TQ_i=I$.  Let
  $$
  Q=\left[
  \begin{array}{c;{2pt/2pt}c;{2pt/2pt}c}
  Q_1 &\b0 &\b0 \\ \hdashline[2pt/2pt]
  \b0 & Q_2 &\b0 \\ \hdashline[2pt/2pt]
 \b0  & \b0& Q_3  \\
  \end{array}
  \right],
  \\\
  D=\left[
  \begin{array}{c;{2pt/2pt}c;{2pt/2pt}c}
  D_1 & \b0&\b0 \\ \hdashline[2pt/2pt]
\b0   & D_2 &\b0 \\ \hdashline[2pt/2pt]
  \b0 &\b0 & D_3  \\
  \end{array}
  \right].
  $$
 Since $Q$ is an orthogonal matrix, eigenvalues of $Q^TAQ$ are same as those of $A$. We have
 \begin{eqnarray}\label{rank-two-modified} Q^TAQ=Q^T(T+\beta_1\bw_1\bw_1^T+\beta_2\bw_2\bw_2^T)Q
=D+\beta_1\bv_1\bv_1^T+\beta_2\bv_2\bv_2^T,
\end{eqnarray}
 where $\bv_1=Q^T\bw_1$, $\bv_2=Q^T\bw_2$.
 Then
 \begin{equation}\label{v1v2}
 \bv_1=\begin{pmatrix} \text{last column of }Q_1^T\\\text{first column of }Q_2^T \\ \b0\end{pmatrix},
 \bv_2=\begin{pmatrix}\b0\\\text{last column of }Q_2^T\\ \text{first column of }Q_3^T\end{pmatrix}.
 \end{equation}
 Let $\bx$ be an eigenvector of $Q^TAQ$ corresponding to an eigenvalue $\lambda$. Then we have
 \begin{eqnarray*}
(D+\beta_1\bv_1\bv_1^T+\beta_2\bv_2\bv_2^T)\bx=\lambda\bx,
\end{eqnarray*} which is equivalent to
  \begin{eqnarray}\label{derive-deflation2}
 (D-\lambda I)\bx =-\beta_1(\bv_1^T\bx)\bv_1-\beta_2(\bv_2^T\bx)\bv_2.
\end{eqnarray}
 If $(D-\lambda I)$ is invertible, $\bx$ can be expressed as a linear combination of  $\bu_1$ and $\bu_2$ where $\bu_1=(\lambda I -D)^{-1}\bv_1$, $\bu_2=(\lambda I -D)^{-1}\bv_2$.
 Using this  expression of $\bx$, we can derive the secular equation (see Section \ref{secularequation}).

Note that the matrix  $(D-\lambda I)$ is singular if and only if the eigenvalues of $A$ coincide with the diagonals of $D$.
 In order to exclude those cases, we deflate them by computing the corresponding eigenvectors. This will be
 explained in detail in the next section.

\section{Deflation} \label{deflation step} In this section we will determine the cases when some of the diagonal entries of $D$ are
  eigenvalues  of $A$.
Expressing equation (\ref{derive-deflation2})  componentwise, we see
 \begin{eqnarray}\label{derive-deflation3}\begin{pmatrix}(d_1-\lambda)x_1\\(d_2-\lambda)x_2\\
\vdots\\(d_n-\lambda)x_n\end{pmatrix}+\begin{pmatrix}\beta_1(\bv_1^T\bx)v_{11}\\
 \beta_1(\bv_1^T\bx)v_{12}\\\vdots\\\beta_1(\bv_1^T\bx)v_{1n}\end{pmatrix}
 +\begin{pmatrix}\beta_2(\bv_2^T\bx)v_{21}\\\beta_2(\bv_2^T\bx)v_{22}\\\vdots\\
 \beta_2(\bv_2^T\bx)v_{2n}\end{pmatrix}=\b0.\end{eqnarray}

We will check row by row whether each diagonal entry is an eigenvalue or not.
We first assume $d_i$ is an eigenvalue. Based on the values of $v_{1i}, v_{2i}$, we can divide it into four cases as follows:

\begin{enumerate}
\item Assume $v_{1i}=0$ and $v_{2i}=0$. Then $\be_i$ is  an eigenvector  for $d_i$.

\item Assume $v_{1i}=0$ and $v_{2i}\ne0$. Since $\beta_2\ne0$, we have $\bv_2^T\bx=0$. Hence,
deleting the $i$-th row in
equation (\ref{derive-deflation3}), we obtain  the following equation.
\begin{equation}\label{rank-one-modificiation}
 (D'-\lambda I)\bx'+\beta_1(\bv'_1(\bv_1')^T)\bx'=0,
\end{equation} where $D', \bv'_1, \bx'$ are obtained by omitting the $i$-th entry.
In this case we can proceed exactly as in Cuppen's divide and conquer method to check whether $d_i$ is an eigenvalue or not. This algorithm is summarized as follows. \cite{Cuppen}
\begin{enumerate}
\item If there exists $j$ such that $d_j=d_i, v_{1j}=0$, then $d_i$ is an eigenvalue of $D'$ (and hence of $D$) and $\bx'=\be_j'$ where $\be_j'$ is obtained from  $\be_j$ by omitting the
    $i$-th entry.

\item If there exists $j, k$ such that $d_j=d_k=d_i$, then $d_i$ is an eigenvalue and
$$\bx'= \begin{pmatrix}
\left[ \begin{array}{c} 0\\\vdots\\-v_{1k}\\\vdots\\v_{1j}\\ \vdots\\0 \end{array} \right] & \begin{array}{c}jth \\ \\ kth\\ \end{array}  \end{pmatrix} ',$$
where $\prime$ is obtained by omitting the $i$-th entry.
\item If there is no $j$ such that $d_j=d_i$, we obtain the secular equation of (\ref{rank-one-modificiation}) from the Cuppen's divide and conquer method as follows:
 \begin{equation}\label{dis-f-11}
   f_{11}(\lambda):=1-\beta_1 \displaystyle\sum_{q\ne i} \frac{v_{1q}^2}{(\lambda-d_q)}.
 \end{equation}
If $f_{11}(d_i)=0$, then $d_i$ is an eigenvalue and $x'_r=\frac{v_{1r}}{d_r-d_i}$ for all $r\ne i$.
Hence we define $ f_{11}(\lambda)$ as a discriminant for $d_i$.
\end{enumerate}
For all three cases, we can find $x_i$ by solving  $\bv_2^T\bx=0$ with given $\bx'$ above. Therefore, $d_i$ is an eigenvalue and the corresponding eigenvector is $\bx$.
\item Assume $v_{1i}\ne0$ and $v_{2i}=0$. We can apply the same method as above and obtain the following a discriminant:
\begin{equation}\label{dis-f-12}
   f_{12}(\lambda):=1-\beta_2 \displaystyle\sum_{q\ne i} \frac{v_{2q}^2}{(\lambda-d_q)}.
 \end{equation}

\item Finally assume $v_{1i}\ne0$ and $v_{2i}\ne0$. Suppose the diagonal entry $d_i$ is repeated  exactly $p$ times. Rearranging the indices, let us assume  $d_{m_1}=d_{m_2}=\dots=d_{m_p}$.
From equation (\ref{derive-deflation3}), we see the $j$-th row, for $j=m_1,\cdots, m_p$, satisfy
the following equations:
  \begin{equation} \label{tri-di}
 \left\{
  \begin{matrix}
\beta_1(\bv_1^T\bx)v_{1m_1}+\beta_2(\bv_2^T\bx)v_{2m_1}=0\\\vdots\\
\beta_1(\bv_1^T\bx)v_{1m_p}+\beta_2(\bv_2^T\bx)v_{2m_p}=0.
\end{matrix}
\right .
\end{equation}
In matrix form, this can be written as   $$C\bz=\b0, $$ where
$$C:=\begin{bmatrix} v_{1m_1} & v_{2m_1} \\  v_{1m_2} & v_{2m_2} \\  \vdots & \vdots \\  v_{1m_p} & v_{2m_p} \end{bmatrix} \mbox{ is $p\times 2$ and } \bz=\begin{bmatrix} \beta_1(\bv_1^T\bx)\\ \beta_2(\bv_2^T\bx) \end{bmatrix}\mbox{ is $2\times 1$ vector. }  $$
Depending on the rank of $C$, these are divided into two cases:
\begin{enumerate}
\item Assume $rank(C)=2$. Then equation (\ref{tri-di}) has the trivial solution, which means $\bv_1^T\bx=0$ and $\bv_2^T\bx=0$.  Since $(d_q -d_i) \ne 0$, we see from equation (\ref{derive-deflation3}) that $x_q=0$, for all $q\ne m_1, m_2 \dots m_p$. Therefore,
 \begin{equation} \label{eq-v1v2}
 \left\{
  \begin{matrix}
\bv_1^T\bx=v_{1m_1}x_{m_1}+v_{1m_2}x_{m_2}+\dots+v_{1m_p}x_{m_p}=0\\
\bv_2^T\bx=v_{2m_1}x_{m_1}+v_{2m_2}x_{m_2}+\dots+v_{2m_p}x_{m_p}=0.
\end{matrix}
\right.
\end{equation}
From these equations, we can find exactly $p-2$ independent eigenvectors.
In this case we still have two identical diagonal entries which are not deflated.

\item Assume $rank(C)=1$. Consider the case when $\bv_1^T\bx=0$.  In this case it follows that $\bv_2^T\bx=0$ and we can find exactly $p-1$ nontrivial solutions associated with eigenvalue $d_i$.
Since $x_q=0$ for every $q\ne m_1, m_2, \dots, m_p$,  by the same way as above  we can deflate the corresponding $p-1$ rows and columns. After the deflation, we would have only one diagonal entry of $D$ which is equal to $d_i$.
Without loss of generality, we may assume that this diagonal entry  is $d_{m1}$.

Next consider the case when  $\bv_1^T\bx\ne 0$. In this case, it follows that $\bv_2^T\bx \ne 0$. Fix any such $\bx$. Then we have
\begin{equation} \label{v2x2}
\bv_2^T\bx =l,
\end{equation} for some $l\ne0$.
From equation (\ref{tri-di}), we obtain
\begin{equation} \label{v1tx}
\bv_1^T\bx =-\frac{\beta_2v_{2m_1}(\bv_2^T\bx)}{\beta_1v_{1m_1}}=-\frac{\beta_2 l v_{2m_1}}{\beta_1v_{1m_1}}.
\end{equation}

Substitute this value of $\bv_1^T\bx$  into equation (\ref{derive-deflation3}), we get
\begin{equation} \label{xl}
x_q =\frac{(-\frac{\beta_2 l v_{2m_1}}{v_{1m_1}})v_{1q}+\beta_2 l v_{2q}}{d_{m_1}-d_q}=\frac{\beta_2 l}{v_{1m_1}}  \cdot \frac{v_{1m_1}v_{2q}-v_{2m_1}v_{1q}}{d_{m_1}-d_q}
\end{equation}
for all $q\ne m_1$. Solving equation (\ref{v1tx}) using these $x_q$, we obtain $x_{m_1}$ as

\begin{equation} \label{x1}
\begin{array}{l l}
x_{m_1} &=-\frac{\beta_2 l v_{2m_1}}{\beta_1v_{1m_1}^2}-\frac{1}{v_{1m_1}} \displaystyle\sum_{q\ne m_1} v_{1q}x_q\\&=-\frac{\beta_2 l v_{2m_1}}{\beta_1v_{1m_1}^2}-\frac{\beta_2 l}{v_{1m_1}^2} \displaystyle\sum_{q\ne m_1} v_{1q}\frac{v_{1m_1}v_{2q}-v_{2m_1}v_{1q}}{d_{m_1}-d_q}.
\end{array}
\end{equation}

Substituting the values  $x_q$ in (\ref{xl}) and (\ref{x1}) for all  $q=1,\cdots,$  into the equation (\ref{v2x2}),
$$
\begin{array}{l l}
l&=\bv_2^T\bx\\&=\displaystyle\sum_{q\ne m_1} v_{2q}x_q + v_{2m_1} x_{m_1}\\
 &=\left(\frac{\beta_2 l}{v_{1m_1}} \displaystyle\sum_{q\ne m_1}  v_{2q} \frac{v_{1m_1}v_{2q}-v_{2m_1}v_{1q}}{d_{m_1}-d_q}\right)+\\&\quad\left( -\frac{\beta_2 l v_{2m_1}^2}{\beta_1v_{1m_1}^2}-\frac{\beta_2 l v_{2m_1}}{v_{1m_1}^2} \displaystyle\sum_{q\ne m_1} v_{1q}\frac{v_{1m_1}v_{2q}-v_{2m_1}v_{1q}}{d_{m_1}-d_q}\right)\\
&= -\frac{\beta_2 l v_{2m_1}}{\beta_1v_{1m_1}^2} +\frac{\beta_2 l}{v_{1m_1}^2}\displaystyle\sum_{q\ne m_1} \frac{v_{1m_1}^2v_{2q}^2 -2v_{1m_1}v_{2m_1}v_{1q}v_{2q}+v_{2m_1}^2v_{1q}^2}{d_{m_1}-d_q}.
\end{array}
$$
Since $l\ne0$, we obtain
$$
1+\frac{\beta_2 v_{2m_1}}{\beta_1v_{1m_1}^2} -\frac{\beta_2}{v_{1m_1}^2}\displaystyle\sum_{q\ne m_1} \frac{v_{1m_1}^2v_{2q}^2 -2v_{1m_1}v_{2m_1}v_{1q}v_{2q}+v_{2m_1}^2v_{1q}^2}{d_{m_1}-d_q}=0.
$$

Therefore, we define the following function as the discriminant in this case:
$$f_{2}(\lambda):=\beta_1 v_{1m_1}^2 + \beta_2 v_{2m_1}^2-\beta_1\beta_2  \displaystyle\sum_{q\ne m_1} \frac{(v_{2m_1}v_{1q}-v_{1m_1}v_{2q})^2}{\lambda-d_q}.$$

If $f_{2}(d_i)=0$, then $d_i$ is an eigenvalue and the corresponding eigenvector is $\bx$ given in    (\ref{xl}) and  (\ref{x1}). Otherwise, $d_i$ cannot be an eigenvalue.
\end{enumerate}

\end{enumerate}

Since we have gone through every possible cases, $d_i$ cannot be an eigenvalue after deflation.

 \renewcommand{\labelenumi}{(\arabic{enumi})}
\section{Secular Equation}\label{secularequation}
We now assume that we have deflated all the cases as above.
\begin{theorem}
The eigenvalues of $$D+\beta_1\bv_1\bv_1^T+\beta_2\bv_2\bv_2^T$$ are the roots of the secular equation defined by
\begin{eqnarray}\begin{aligned}\label{secular}
f(\lambda)&: =&\beta_1\beta_2\displaystyle\sum_{q=1}^n \sum_{r= q+1}^n\frac{(v_{1q}v_{2r}-v_{1r}v_{2q})^2}{(\lambda -d_q)(\lambda -d_r)} - \displaystyle\sum_{q=1}^n \frac{\beta_1v_{1q}^2+\beta_2v_{2q}^2}{\lambda -d_q}+1.
\end{aligned}\end{eqnarray}
\end{theorem}
\begin{proof}
Since the $d_i$ is not an eigenvalue for all $i$, $(D-\lambda I)$ is invertible. From (\ref{derive-deflation2}), we have
\begin{eqnarray}\label{eigenvector} \bx &= &(\lambda I -D)^{-1}(\beta_1(\bv_1^T \bx)\bv_1+\beta_2(\bv_2^T\bx)\bv_2)\\ \label{eigenvector2}&:= &a\bu_1+b\bu_2, \end{eqnarray}
where $a=\beta_1(\bv_1^T\bx)$, $b=\beta_2(\bv_2^T\bx)$, $\bu_1=(\lambda I -D)^{-1}\bv_1$, $\bu_2=(\lambda I -D)^{-1}\bv_2$. Substituting (\ref{eigenvector2})  into (\ref{derive-deflation2}), we obtain
\begin{eqnarray}\label{to-secular}(-a+a\beta_1 c_1 + b\beta_1 c_3) \bv_1 + (-b + b\beta_2 c_2 + a \beta_2 c_3 ) \bv_2=0,
\end{eqnarray}
where $c_1=\bv_1 ^T \bu_1$, $c_2=\bv_2 ^T \bu_2$, $c_3=\bv_1 ^T \bu_2=\bv_2 ^T \bu_1$.

If $\bv_1$ is multiple of $\bv_2$, this problem is reduced to the original divide and conquer method of Cuppen.
Therefore, we assume that $\bv_1$ is not multiple of $\bv_2$. Hence  we have
\begin{eqnarray}\label{equation1} -a+a\beta_1 c_1 + b\beta_1 c_3 =0, \quad -b + b\beta_2 c_2 + a \beta_2 c_3 =0.
 \end{eqnarray} Since $a(1-\beta_1 c_1)=b\beta_1 c_3 $ and $b(1- \beta_2 c_2)= a \beta_2 c_3$, by multiplying we obtain
 \begin{eqnarray}  \label{equation2} ab(1-\beta_1 c_1 )(1- \beta_2 c_2)=ab\beta_1 c_3\beta_2 c_3.\end{eqnarray}
First we assume $ab\ne0$. Then we can cancel $ab$ and we obtain
\begin{eqnarray}\label{pre-secular}
(-1+\beta_1c_1)(-1+\beta_2c_2)=\beta_1 \beta_2 c_3^2.
\end{eqnarray}
Here, we define  the secular equation $f(\lambda)$ as follows:
\begin{eqnarray} \label{secular2}
f(\lambda)&:=\beta_1\beta_2(c_1 c_2 - c_3^2)-\beta_1c_1-\beta_2c_2 +1.
\end{eqnarray} By substituting values of $c_1, c_2$ and $c_3$, we obtain
\begin{eqnarray}\begin{aligned}\label{secular3}
f(\lambda)&=\beta_1\beta_2\displaystyle\sum_{q=1}^n \frac{v_{1q}^2}{\lambda -d_q}\displaystyle\sum_{q=1}^n \frac{v_{2q}^2}{\lambda -d_q}-\beta_1\beta_2\left(\displaystyle\sum_{q=1}^n \frac{v_{1q}v_{2q}}{\lambda -d_q}\right)^2\\\ &\ \ -\displaystyle\sum_{q=1}^n \frac{\beta_1v_{1q}^2+\beta_2v_{2q}^2}{\lambda -d_q}+1
\\&=\beta_1\beta_2\displaystyle\sum_{q=1}^n \sum_{r= q+1}^n\frac{(v_{1q}v_{2r}-v_{1r}v_{2q})^2}{(\lambda -d_q)(\lambda -d_r)} - \displaystyle\sum_{q=1}^n \frac{\beta_1v_{1q}^2+\beta_2v_{2q}^2}{\lambda -d_q}+1.
\end{aligned}\end{eqnarray}If $a=0, b\ne0$, we see $c_3=0$ and $1-\beta_2 c_2 =0$. If $a\ne0, b=0$, we see $c_3=0$ and $ 1-\beta_1c_1=0$.
We see these   cases also satisfy the secular equation (\ref{secular3}).
\end{proof}
We note that the  derivative of $f(\lambda)$ is
$$
f^\prime(\lambda)= -\beta_1\beta_2\displaystyle  \sum_{q=1}^n \sum_{r=q+1}^n \frac{(2\lambda-d_q-d_r)(v_{1q}v_{2r}-v_{1r}v_{2q})^2}{(\lambda -d_q)^2(\lambda -d_r)^2} + \displaystyle\sum_{q=1}^n \frac{\beta_1v_{1q}^2+\beta_2v_{2q}^2}{(\lambda -d_q)^2}.
$$

\section{Computing eigenvalues from secular equation} \label{eigenvalue section}
 Since triple or more identical diagonal entries are deflated by Section 3,  cases (1), (2)-(b), (3)-(b), (4)-(a),(b), there can be at most two identical diagonal entries for each diagonal value of $D$.
Fix an index $i=i_0$.
If there exists an index $j_0\ne i_0$ such that $d_{j_0}=d_{i_0}$, we  call this $d_{i_0}$   multiple.
Otherwise, we call $d_{i_0}$   simple. \newline
First, we will check the sign of  $\displaystyle\lim_{\lambda \rightarrow d_{i_0}^{+}} f(\lambda)$ and
 $\displaystyle\lim_{\lambda \rightarrow d_{i_0}^{-}} f(\lambda)$ which are determined by
the sign of $f(d_{i_0}+\epsilon)$ for small $\epsilon$. By checking the dominating term in (\ref{secular2}),
we see it has the same sign as
$$\left\{
  \begin{array}{l l}
      \beta_1\beta_2\displaystyle\sum_{\substack{q=1 \\q\ne i_0}}^n \frac{(v_{1q}v_{2i_0}-v_{1i_0}v_{2q})^2}{(d_{i_0}-d_q+\epsilon)\epsilon} -\frac{\beta_1v_{1{i_0}}^2+\beta_2v_{2i_0}^2}{\epsilon},& \quad d_{i_0} \text{ simple}\\
  2\beta_1\beta_2\frac{(v_{1j_0}v_{2i_0}-v_{1i_0}v_{2j_0})^2}{\epsilon^2},& \quad d_{i_0} \text{ multiple.}
  \end{array} \right.
  $$
For  multiple $d_{i_0}$, we know that $(v_{1j_0}v_{2i_0}-v_{1i_0}v_{2j_0})^2\ne0$ from (2)-(a), (3)-(a), (4)-(b) of Section 3.
Hence the sign of  $\displaystyle\lim_{\lambda \rightarrow d_{i_0}^{+}} f(\lambda)$ and $\displaystyle\lim_{\lambda \rightarrow d_{i_0}^{-}} f(\lambda)$ will be determined by the coefficients of $\frac1 {\epsilon}$ and
  $\frac1 {\epsilon^2}$, which are given as follows:
\begin{eqnarray}\label{sign-of-fp} g^+_{i_0}:=\left\{
  \begin{array}{l l}
      -\beta_1\beta_2\displaystyle\sum_{\substack{q=1 \\q\ne i_0}} ^n\frac{(v_{1q}v_{2,i_0}-v_{1,i_0}v_{2q})^2}{d_q-d_{i_0}} -\beta_1v_{1i_0}^2-\beta_2v_{2i_0}^2, & \text{$d_{i_0}$ simple}\\
\beta_1\beta_2,  & \text{$d_{i_0}$ multiple.}
  \end{array} \right.
  \end{eqnarray}
  \begin{eqnarray}\label{sign-of-fm} g^-_{i_0}:=\left\{
  \begin{array}{l l}
     \beta_1\beta_2\displaystyle\sum_{\substack{q=1 \\q\ne i_0}}^n \frac{(v_{1q}v_{2,i_0}-v_{1,i_0}v_{2q})^2}{d_q-d_{i_0}}+\beta_1v_{1i_0}^2+\beta_2v_{2i_0}^2, &  \text{$d_{i_0}$ simple}\\
  \beta_1\beta_2,  &  \text{$d_{i_0}$ multiple.}
  \end{array} \right.
  \end{eqnarray}

Now, we are going to determine the number of roots for each interval $(d_i,d_{i+1})$.
First, we are going to deal with the cases without multiple diagonal entries (Section \ref{sdc}).
After that, we are going to deal with the cases having some multiple diagonal entries (Section \ref{mdc}).

\subsection{Cases without multiple diagonal  entries} \label{sdc}
 Suppose that $(n-m)$ diagonal entries are deflated in the steps (1) - (4) of  Section 3.
For the simplicity of presentation, let us sort  and reindex the remaining $m$ diagonal entries in an ascending order as follows:
$$0=d_0<d_1<\dots<d_m<d_{m+1}=\infty.$$
 We let $I_q:=(d_q,d_{q+1})$ and define $I_{qr}$ as the union of $(r-q)$ intervals as
 $$ I_{qr}:=\displaystyle\bigcup_{i=q}^{r-1} I_i = (d_{q},d_{q+1})\cup (d_{q+1},d_{q+2}) \cup \dots \cup (d_{r-1},d_r),$$
  where $q=0,1,2,\dots,m$ and $r=q+1,\dots,m+1$.
 In addition, we define the following matrices of sizes $k_1+k_2$ and  $k_2+k_3$ respectively,
 \begin{equation}\label{Bi}
 B_{i}:=\left[
  \begin{array}{c;{2pt/2pt}c}
  D_i &\\ \hdashline[2pt/2pt]
   & D_{i+1}
  \end{array}
  \right] +\beta_i \bv'_i\bv_i^{\prime T},\quad  i=1, 2
  \end{equation} where $\bv'_i$ is obtained by omitting zero vector part from (\ref{v1v2}).
   Then we see
  \begin{equation}\label{A}
  A=\left[
  \begin{array}{c;{2pt/2pt}c}
  B_1 &\\ \hdashline[2pt/2pt]
   & D_3
  \end{array}
  \right] +\beta_2 \bv_2\bv_2^T
  =\left[
  \begin{array}{c;{2pt/2pt}c}
  D_1 &\\ \hdashline[2pt/2pt]
   & B_2
  \end{array}
  \right] +\beta_1 \bv_1\bv_1^T.
  \end{equation}
   We assume the following spectral decomposition of $B_i$:
   $$B_i=Q'_iD'_iQ_i^{\prime T},\quad Q_i^{\prime T}Q_i'=I,\quad i=1,2$$
and define
$$R_1 =\left[
  \begin{array}{c;{2pt/2pt}c}
  Q'_1&\\ \hdashline[2pt/2pt]
   & I
  \end{array}
  \right] ,\quad
  R_2  =\left[
  \begin{array}{c;{2pt/2pt}c}
  I &\\ \hdashline[2pt/2pt]
   & Q'_2
  \end{array}
  \right].$$
Since $R_1$ and $R_2$ are orthogonal matrices, $A$ has same eigenvalue as $R_1^T A R_1$ and $R_2^T A R_2$:
\begin{equation} \label{rar}
R_1^T A R_1=\left[
  \begin{array}{c;{2pt/2pt}c}
  D'_1&\\ \hdashline[2pt/2pt]
   & D_3
  \end{array}
  \right]+\beta_2 \bz_2\bz_2^T, \quad R_2^T A R_2=\left[
  \begin{array}{c;{2pt/2pt}c}
  D_1&\\ \hdashline[2pt/2pt]
   & D'_2
  \end{array}
  \right]+\beta_1 \bz_1\bz_1^T
\end{equation} where $\bz_2=R_1^T\bv_2, \bz_1=R_2^T\bv_1$.
Also, we define the interval $I'_j$ for $j=1,2,\cdots,m-1$as $$I'_j:=(d''_j,d''_{j+1})$$
where $d''_j$ are diagonal entries from $D'_1$ and $D_3$ (or $D_1$ and $D'_2$) in an ascending  order. Since the
 matrices in (\ref{rar}) are in the rank one modification form, we see from the Cuppen's divide and conquer algorithm that there exists exactly one eigenvalue of $R^T_iAR_i$ (and hence of $A$) in each $I'_j$.

 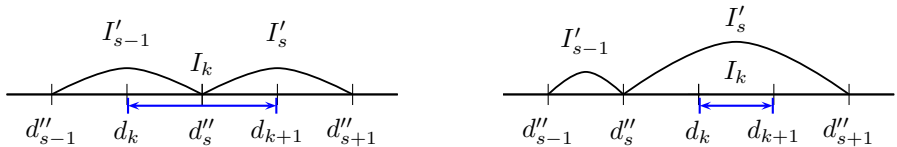
\begin{figure}[ht]
\begin{pspicture}(-3.5,-0.7)(10.2,1.2)
\psaxes[linewidth=0.75\pslinewidth,labels=none](0,0)(-2.6,0)(2.6,0)
\psaxes[linewidth=0.75\pslinewidth,labels=none](6.6,0)(4,0)(9.2,0)
\rput(-2.0,-0.5) {$d''_{s-1}$}
\rput(0.0,0.4) {$I_{k}$}
\rput(0.0,-0.5) {$d''_{s}$}
\rput(2.0,-0.5) {$d''_{s+1}$}
\rput(1.0,-0.5) {$d_{k+1}$}
\rput(-1.0,-0.5) {$d_{k}$}

\rput(-1,0.8) {$I'_{s-1}$}
\rput(1.0,0.8) {$I'_{s}$}
\rput(4.6,-0.5) {$d''_{s-1}$}
\rput(0,-.15)
 {\pcline[linewidth=0.8pt,linecolor=blue]{|<->|}(-1.,0)(1,0)}
\rput(7.1,0.3) {$I_{k}$}
\rput(6.6,-0.5) {$d_{k}$}
\rput(8.6,-0.5) {$d''_{s+1}$}
\rput(7.6,-0.5) {$d_{k+1}$}
\rput(5.6,-0.5) {$d''_{s}$}
\rput(5.1,0.7) {$I'_{s-1}$}
\rput(7.1,1) {$I'_{s}$}
\pscurve[linewidth=0.7pt](-2.0,0)(-1,0.35)(0,0)
\pscurve[linewidth=0.7pt](0,0)(1,0.35)(2,0)
\pscurve[linewidth=0.7pt](4.6,0)(5.1,0.3)(5.6,0)
\pscurve[linewidth=0.7pt](5.6,0)(7.1,0.7)(8.6,0)
\rput(0,-.15)
 {\pcline[linewidth=0.8pt,linecolor=blue]{|<->|}(6.6,0)(7.6,0)}\end{pspicture}
\caption{$I_k$ has at most two eigenvalues when $d''_{s}$ is in $I_k$ (left) and $I_k$ has at most one eigenvalue when $d''_{s}$ is not in $I_k$ (right)}\label{fig-lemma1}
\end{figure}
 \begin{lemma}  \label{interval-jk}
  Interval $I_{j,j+k}$ can have at least $k-1$ roots and at most $k+1$ roots for $j=1,2,\cdots,m-1$ and $k=1,2,\cdots,m-j$.
 \end{lemma}
 \begin{proof}
 Fix the index $j=j_0$ and $k=k_0$ so that the $I_{j_0,j_0+k_0}$ has $k_0+1$ diagonals including both ends.
 If $d_{j_0+k_0}$ is from $D_1$ or $D_2$, choose $B_i=B_1$.
 Otherwise, choose $B_i=B_2$. Assume that $B_i$ has $l$ diagonal entries. Depending on whether $d_{j_0}$ is from those $l$ diagonals or not, we can divide in into two cases.
 \begin{enumerate}
 \item Assume that $d_{j_0}$ is one of the diagonals of $B_i$.
 Then from (\ref{Bi}) we conclude by CDC algorithm that $B_i$ has exactly $l-1$ eigenvalues in $I_{j_0,j_0+k_0}$.
 Since each of $l-1$ eigenvalues is one of $d''$, there will be at least $(l-1)+(k_0+1-l)$ such $d''$s in $I_{j_0,j_0+k_0}$. ($k_0+1-l$ $d''$'s are from $D_3$ or $D_1$). Then we have
 \begin{equation} d''_{s-1} <d_{j_0}<d''_{s}<\cdots<d''_{s+k_0-1}<d_{j_0+k_0}<d''_{s+k_0} \end{equation} for some $s$.
 Since each $I'_q$ has exactly one eigenvalue of $A$ for $q=s-1,s,\cdots,s+k_0-1$, $I_{j_0,j_0+k_0}$ can have at least $k_0-1$ and at most $k_0+1$ eigenvalues of $A$ (see Figure \ref{fig-lemma1})

 \item Assume that $d_{j_0}$ is not one of the diagonals of $B_i$. (i.e., $d_{j_0}\in D_3$ or $d_{j_0}\in D_1$.)
 Then $B_i$ can have $l-1$ or $l$ eigenvalues in $I_{j_0,j_0+k_0}$ by CDC algorithm.
 Since each eigenvalues of $B_i$ is one of $d''$, there will be $(l-1)+(k_0+1-l)$ or $l+(k_0+1-l)$ such $d''$s in $I_{j_0,j_0+k_0}$ including $d_{j_0}$. Then we have
$$\left\{
  \begin{array}{l l}
  d''_{s} =d_{j_0}<d''_{s+1}<\cdots<d''_{s+k_0-1}<d_{j_0+k_0}<d''_{s+k_0} &\, \text{if $k_0$ such $d''$s}\\
  d''_{s} =d_{j_0}<d''_{s+1}<\cdots<d''_{s+k_0}<d_{j_0+k_0}<d''_{s+k_0+1} &\,  \text{if $k_0+1$ such $d''$s}
  \end{array} \right.
  $$for some $s$.
 Then $I_{j_0,j_0+k_0}$ can have $k_0-1$ or $k_0$ eigenvalues of $A$ for the first case and  can have $k_0$ or $k_0+1$ eigenvalues of $A$ for the second case.
 \end{enumerate}
 In both cases  $I_{j_0,j_0+k_0}$ has $k_0-1$ or $k_0$ or $k_0+1$ eigenvalues of $A$, which completes the proof
  of Lemma.
  \end{proof}
From this Lemma, we can say that there can be at most two eigenvalues of $A$ in each interval $I_j, j=1,\cdots,m-1$.
Let us classify every intervals into two groups $S^{+}$ and $S^{-}$ as follows:

$$ \left\{
  \begin{array}{l l}
     I_{jk} \in S^{+}, & \quad \text{if $g^+_j g^-_k>0$}\\
  I_{jk} \in S^{-}, & \quad \text{if $g^+_j g^-_k<0$}
  \end{array}. \right.
$$
 Since every interval can have at most two roots by Lemma \ref{interval-jk} and the secular equation is continuous on every interval $I_j$, we can conclude that every interval in $S^-$ has one root by intermediate value theorem.
 Clearly, every interval $I_j$ in $S^+$ has two roots or no root.
 By using mathematical induction, we can prove the following lemma.

  \begin{lemma}  \label{interval-s}
  For any interval $I_{j,j+k}$ in $S^{-}$ has exactly $k$ roots where $j=1,2,\dots,m-1$ and $k\leq m-j$. Also, for any interval $I_{j,j+k}$ in $S^{+}$ has $k-1$ or $k+1$ roots where $j=1,2,\dots,m-1$ and $k\leq m-j$.
  \end{lemma}
 \begin{proof}
We have already seen that each interval $I_j$ satisfies the   lemma.
Next, fix $j=j_0$ and suppose that the statements of the  lemma holds when $k=k_0$.
\begin{enumerate}
\item Assume $I_{j_0,j_0+k_0} \in S^-$ and $I_{j_0+k_0}\in S^-$. Since $g^+_{j_0+k_0}$ and $g^-_{j_0+k_0}$ has opposite sign, \\$I_{j_0,j_0+k_0+1} \in S^-$.
By induction, there will be $k_0$ roots in $I_{j_0,j_0+k_0}$ and one root in $I_{j_0+k_0}$.
Therefore, $I_{j_0,j_0+k_0+1}$ will have $k_0+1$ roots and $k=k_0+1$ satisfies the lemma.
\item Assume $I_{j_0,j_0+k_0} \in S^-$ and $I_{j_0+k_0}\in S^+$.
Then we have $I_{j_0,j_0+k_0+1} \in S^+$.
By induction, there will be $k_0$ roots in $I_{j_0,j_0+k_0}$ and two or no roots in $I_{j_0+k_0}$.
Therefore, $I_{j_0,j_0+k_0+1}$ will have $k_0$ or $k_0+2$ roots and $k=k_0+1$ satisfies the lemma.
\item Assume $I_{j_0,j_0+k_0} \in S^+$ and $I_{j_0+k_0}\in S^-$.
Then we have $I_{j_0,j_0+k_0+1} \in S^+$.
By induction, there will be $k_0-1$ or $k_0+1$ roots in $I_{j_0,j_0+k_0}$ and one root in $I_{j_0+k_0}$.
Therefore, $I_{j_0,j_0+k_0+1}$ will have $k_0$ or $k_0+2$ roots and $k=k_0+1$ satisfies the lemma.
\item Assume $I_{j_0,j_0+k_0} \in S^+$ and $I_{j_0+k_0}\in S^+$.
Then we have  $I_{j_0,j_0+k_0+1} \in S^-$.
By induction,  there will be $k_0-1$ or $k_0+1$ roots in $I_{j_0,j_0+k_0}$ and two or no roots in $I_{j_0+k_0}$.
Then $I_{j_0,j_0+k_0+1}$ can have $k_0-1$ or $k_0+1$ or $k_0+3$ roots.
However, this interval only can have $k_0+1$ roots by Lemma \ref{interval-jk} and $k=k_0+1$ satisfies the lemma.
\end{enumerate}
Hence the statements of the  lemma holds when $k=k_0+1$.
 \end{proof}
Since the secular equation can have at most $n$ roots, from Lemma \ref{interval-jk} we see that $I_0\cup I_m$ can have at most two roots.
 \renewcommand{\labelenumi}{(\arabic{enumi})}

 \begin{theorem}  \label{thm-num-of-roots}
 Number of roots for each interval is determined by the sign of $g_i^-$ for $i=1,2,\dots,m$.
  \end{theorem}
 \begin{proof}
 For each interval $I_j$, $j=1,2,\dots m-1$, we can determine the exact number of roots by the following:
 If $I_j \in S^-$, there will be only one root in $I_j$.
 If $I_j \in S^+$, there can be at most two eigenvalues by Lemma \ref{interval-jk}. We can determine the exact number of roots by evaluating the sign of $f(\lambda_0)$ where $\lambda_0$ is a solution of $f'(\lambda)=0$.
 If the sign of $f(\lambda_0)$ is the same as those of $f(d_j^+)$ and $f(d_{j+1}^-)$, there will be no root.
 On the other hand, if the sign of $f(\lambda_0)$ is  the opposite of $f(d_j^+)$ and $f(d_{j+1}^-)$, there will be two roots in $I_j$.
 Therefore, we know the number of roots in $I_{1m}$ since $I_{1m}=I_{1}\cup \cdots I_{m-1}$.
 Now, we determine the number of roots in $I_0$ and $I_m$.
  Depending on the sign of $f(d_1^+)$ and $f(d_m^-)$, we can divide it into four cases as follows:
  \begin{enumerate}
  \item If $g_1^+>0$, $g_m^->0$, we see $I_{1m}\in S^+$ and $g_m^+<0$.
  Since $\lim_{\lambda \rightarrow \infty} f(\lambda)=1$, there exists one root in $I_m$.
  Since $I_{1m}$ can have $m-2$ or $m$ roots by Lemma \ref{interval-s},  $I_{1m}$ has $m-2$ roots.
  Therefore, $I_0$ has one root.
  \item If $g_1^+<0$, $g_m^->0$, we see $I_{1m}\in S^-$ and $g_m^+<0$.   Since $\lim_{\lambda \rightarrow \infty} f(\lambda)=1$, there exists one root in $I_m$.
   Since $I_{1m}$ has $m-1$ roots by Lemma \ref{interval-s}, $I_0$ has no root.

  \item If $g_1^+>0$, $g_m^-<0$, we see $I_{1m}\in S^-$ and $g_m^+>0$.  Again by the fact  $\lim_{\lambda \rightarrow \infty} f(\lambda)=1$ and  $g_m^+>0$,  $I_m$ can have two or zero roots.  Since $I_{1m}$ has $m-1$ roots by Lemma \ref{interval-s},  $I_m$ cannot have two roots. Hence  $I_m$  has no root and  $I_0$ must  have one root.

  \item If $g_1^+<0$, $g_m^-<0$, we see $I_{1m}\in S^+$ and $g_m^+>0$.
  Then there can be $m-2$ or $m$ roots in $I_{1m}$ by Lemma \ref{interval-s}.
  Since $I_m$ belongs to $S^+$, there will be two or no roots in $I_m$.
  Hence, by counting, we see $I_0$ also has two or no roots. Depending on the sign of $\beta_1$ and $\beta_2$, this is classified into the following cases:

  \begin{enumerate}
  \item Assume $\beta_1>0$ and $\beta_2>0$. If $d_{m}$ is from $D_1$ or $D_2$, choose $B_i=B_1$.
 Otherwise, choose $B_i=B_2$. Applying the result of CDC to the form (\ref{Bi}), we see $d_{m}<d'_{i,r}$ for some $r$.  It is well known that when we apply the CDC algorithm for (\ref{rank-one-mod}) then there is a root at the left end interval when  $\beta$
  is negative and there is a root at right end interval when  $\beta$ is positive.
 Applying this result of CDC to the form (\ref{Bi}), we see there is at least one root in the interval $(d'_{i,r},d_{m+1})$. This root belongs to $I_m$.  Since $I_m$ can have two or no roots, we see that
  $I_m$ has $2$ roots and hence $I_{1m}$ has $m-2$ roots.
  \item Assume $\beta_1<0$ and $\beta_2<0$. Arguing exactly the same way as in case (a), we conclude that the interval $I_0$ has at least one root. This implies that $I_0$ has $2$ roots and $I_{1m}$ has $m-2$ roots.
  \item Assume $\beta_1\beta_2<0$. Similarly, we can see that neither $I_0$ nor $I_m$ has two roots. This implies that $I_{1m}$ has $m$ roots and there is no root in $I_0$ and $I_m$.
  \end{enumerate}
  \end{enumerate}
 \end{proof}

\begin{figure}
\begin{pspicture}(-2.5,-9)(12.,3) \scriptsize
\psset{xunit=0.7, yunit=0.6}
\psset{linewidth=.5\pslinewidth}

\scriptsize
\rput (2.74,0.3) {$\cdots \cdots \cdots $}
  \rput (1,0.3) {\psframebox*{$d_1 $}}
  \rput (2,0.3) {\psframebox*{$d_2 $}}
  \rput (3.5,0.3) {\psframebox*{$d_{n-1} $}}
  \rput (4.5,0.3) {\psframebox*{$d_{n} $}}
  \rput (2.85,2) {\psframebox*{\large$\cdots$}}
\psaxes[ticks=none,labels=none]{->}(0,0)(-0.3,-4.05)(6.7,4.05)
\psline[linestyle=dashed] (1,-4.0)(1,4.) \psline[linestyle=dashed] (2,-4.0)(2,4.)  \psline[linestyle=dashed] (3.5,-4.0)(3.5,4.)  \psline[linestyle=dashed] (4.5,-4.0)(4.5,4.)
\psplot[linewidth=0.5\pslinewidth, plotpoints=200]{0.195}{0.81}{ 2 x mul 1 sub
 x dup mul x sub div}
\psplot[linewidth=0.5\pslinewidth, plotpoints=200]{1.13}{1.87}{-1 x 1 sub x 2 sub mul div
 5 sub}
 \psplot[linewidth=0.5\pslinewidth, plotpoints=200]{3.672}{4.328}{-1 x 3.5 sub x 4.5 sub mul div
 3 sub}
\psplot[linewidth=0.5\pslinewidth, plotpoints=200]{4.65}{6.7}{-1 x 4.45 sub  div 1 add}
\scriptsize
\rput(3,-5){Typical case of (1)}
\rput(12,-5){Typical case of (2)}
\rput(3,-15){Typical case of (3)}
\rput(12,-15){Typical case of (4)-(b)}
\rput(9,0){\psset{linewidth=.5\pslinewidth}

 \rput (2.74,0.3) {$\cdots \cdots \cdots $}
  \rput (1,0.3) {\psframebox*{$d_1 $}}
  \rput (2,0.3) {\psframebox*{$d_2 $}}
  \rput (3.5,0.3) {\psframebox*{$d_{n-1} $}}
  \rput (4.5,0.3) {\psframebox*{$d_{n} $}}
  \rput (2.85,2) {\psframebox*{\large$\cdots$}}
  \psaxes[ticks=none,labels=none]{->}(0,0)(-0.3,-4.05)(6.7,4.05)
\psline[linestyle=dashed] (1,-4.0)(1,4.) \psline[linestyle=dashed] (2,-4.0)(2,4.)  \psline[linestyle=dashed] (3.5,-4.0)(3.5,4.)  \psline[linestyle=dashed] (4.5,-4.0)(4.5,4.)
\psplot[linewidth=0.5\pslinewidth, plotpoints=200]{0.172}{0.828}{-1 x  x 1 sub mul div
 3 sub}
 \psplot[linewidth=0.5\pslinewidth, plotpoints=200]{3.63}{4.37}{-1 x 3.5 sub x 4.5 sub mul div
 5 sub}
\psplot[linewidth=0.5\pslinewidth, plotpoints=200]{1.195}{1.81}{-1 x 1 sub div -1 x 2 sub div add}

  \psplot[linewidth=0.5\pslinewidth, plotpoints=200]{4.65}{6.7}{-1 x 4.45 sub  div 1 add}
  }

  \rput(0,-10){\psset{linewidth=.5\pslinewidth}
\scriptsize
  \rput (2.74,0.3) {$\cdots \cdots \cdots $}
  \rput (1,0.3) {\psframebox*{$d_1 $}}
  \rput (2,0.3) {\psframebox*{$d_2 $}}
  \rput (3.5,0.3) {\psframebox*{$d_{n-1} $}}
  \rput (4.5,0.3) {\psframebox*{$d_{n} $}}
  \rput (2.85,2) {\psframebox*{\large$\cdots$}}
\psaxes[ticks=none,labels=none]{->}(0,0)(-0.3,-4.05)(6.7,4.05)
\psline[linestyle=dashed] (1,-4.0)(1,4.) \psline[linestyle=dashed] (2,-4.0)(2,4.)  \psline[linestyle=dashed] (3.5,-4.0)(3.5,4.)  \psline[linestyle=dashed] (4.5,-4.0)(4.5,4.)
\psplot[linewidth=0.5\pslinewidth, plotpoints=200]{0.195}{0.81}{ 2 x mul 1 sub
 x dup mul x sub div}
\psplot[linewidth=0.5\pslinewidth, plotpoints=200]{1.13}{1.87}{-1 x 1 sub x 2 sub mul div
 5 sub}
 \psplot[linewidth=0.5\pslinewidth, plotpoints=200]{3.695}{4.31}{1 x 3.5 sub div 1 x 4.5 sub div add}

\psplot[linewidth=0.5\pslinewidth, plotpoints=200]{4.7}{6.7}{1 x 4.45 sub  div }
  }

  \rput(9,-10){\psset{linewidth=.5\pslinewidth}

\rput (2.74,0.3) {$\cdots \cdots \cdots $}
  \rput (1,0.3) {\psframebox*{$d_1 $}}
  \rput (2,0.3) {\psframebox*{$d_2 $}}
  \rput (3.5,0.3) {\psframebox*{$d_{n-1} $}}
  \rput (4.5,0.3) {\psframebox*{$d_{n} $}}
  \rput (2.85,2) {\psframebox*{\large$\cdots$}
   }
  \psaxes[ticks=none,labels=none]{->}(0,0)(-0.3,-4.05)(6.7,4.05)
\psline[linestyle=dashed] (1,-4.0)(1,4.) \psline[linestyle=dashed] (2,-4.0)(2,4.)  \psline[linestyle=dashed] (3.5,-4.0)(3.5,4.)  \psline[linestyle=dashed] (4.5,-4.0)(4.5,4.)
\psplot[linewidth=0.5\pslinewidth, plotpoints=200]{0.13}{0.87}{-1 x  x 1 sub mul div
 5 sub}
\psplot[linewidth=0.5\pslinewidth, plotpoints=200]{3.695}{4.31}{1 x 3.5 sub div 1 x 4.5 sub div add}
\psplot[linewidth=0.5\pslinewidth, plotpoints=200]{1.195}{1.81}{-1 x 1 sub div -1 x 2 sub div add}

  \psplot[linewidth=0.5\pslinewidth, plotpoints=200]{4.7}{6.7}{1 x 4.45 sub  div }

    }
  \end{pspicture}
  \caption{Typical examples of Theorem \ref{thm-num-of-roots}}
\end{figure}
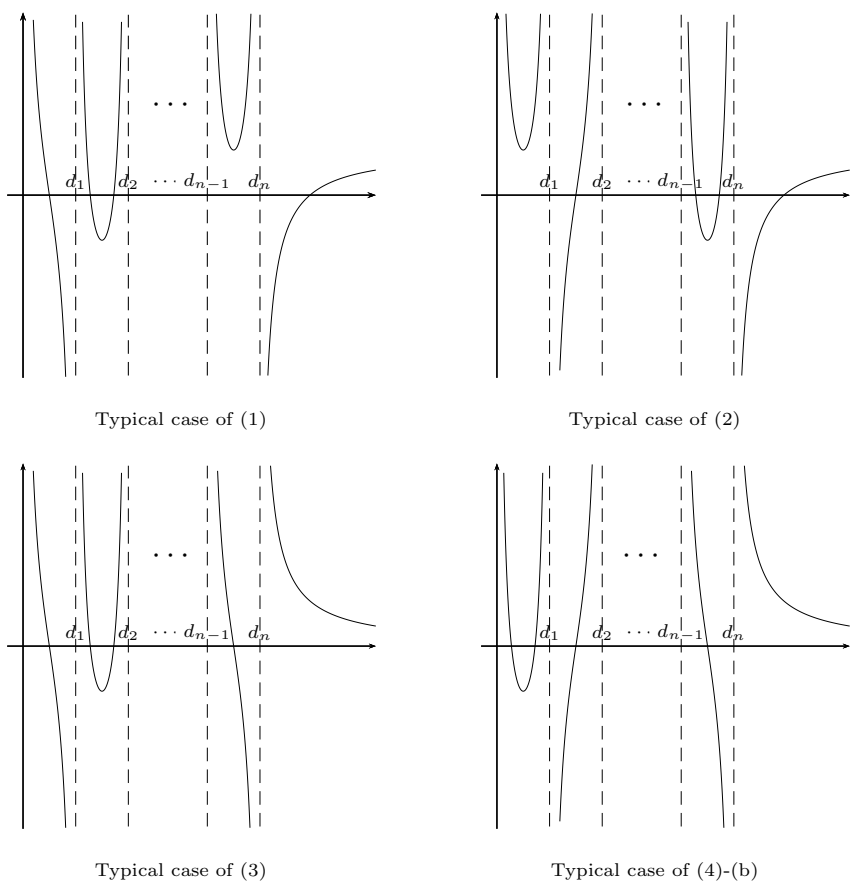
\subsection{Cases with multiple diagonal entries} \label{mdc}
For the multiple diagonal case, the results of previous section can be applied in a similar manner.
Suppose that $d_i=d_{i+1}$ for some $i$. From (\ref{sign-of-fp}) and (\ref{sign-of-fm}), we see  the sign of
 $g_i^+$ and $g_i^-$ are the same.

Again by applying the idea used in the proof of Lemma \ref{interval-jk} and \ref{interval-s}, we see
that there should be one more root in one of $I_{i-1,i}$ or $I_{i+1,i+2}$.

\section{Eigenvector} Since  deflation steps were treated in Section 3, it suffices to
 calculate eigenvectors  using equation (\ref{eigenvector2}). From  (\ref{equation1})  we have
 \begin{equation} \label{equa-eigenv}
 \begin{pmatrix}
  -1+\beta_1 c_1 & \beta_1 c_3\\
\beta_2 c_3 & -1+\beta_2 c_2 \end{pmatrix} \begin{pmatrix}
  a\\ b \end{pmatrix} =\b0.
\end{equation} Since this system has a nontrivial solution, the determinant of the  matrix must be zero, which is nothing but (\ref{pre-secular}).
Eigenvectors corresponding to simple eigenvalues can be obtained by substituting the eigenvalues into  (\ref{equation1}) to solve for the ratio of $a$ to $b$.
For the case of multiple(double) eigenvalues, two linearly independent solutions exist. Hence all the
 entries of the matrix vanish. So any two nonzero numbers $a,b$ are solutions.
Hence we can take $\bu_1$ and $\bu_2$  as the corresponding eigenvectors.

In practice, the eigenvectors computed this way lose orthogonality as in the original Cuppen's divide and conquer algorithm when eigenvalues of $A$ are close to those of $T_i$. To fix these problems, we try two methods:
\subsection{Method 1 - Calculating $\hat{\bv_1}, \hat{\bv_2}$ corresponding to the computed eigenvalues}\label{fixing}
This is a natural extension of Gu and Eisenstat \cite{Gu} fixing the orthogonality problem to the case of rank two modifications.
 However, as it turns out, we run into the lack of equations to find $\bv_1$ and $\bv_2$.  
To see why, first we need the following Lemma which can be found in the
standard textbook.
\begin{lemma} We have
 \begin{enumerate}
\item  $det(A+\beta\bv\bv^T)= det (A)(1+ \beta\bv^T A^{-1}\bv) $
\item  $(A+\beta\bv\bv^T)^{-1}= A^{-1}- \frac{\beta A^{-1} \bv \bv^T A^{-1}}{1+\beta \bv^TA^{-1}\bv} $.
\end{enumerate}
\end{lemma}
 Suppose that we could find two vectors $\hat{\bv_1}, \hat{\bv_2}$  such that $\{\hat\lambda_i\}$ are the exact eigenvalues of the new rank-two modification matrix $\hat D+\beta_1\hat\bv_1\hat\bv_1^T+\beta_2\hat\bv_2\hat\bv_2^T$.
 Proceeding as in Section 3 of \cite{Gu} we are led to investigate the term $\prod_{i=1}^n (\lambda_i - \mu)$.
Using the Lemma, we have
$$
\begin{array}{l l}
&\displaystyle\prod_{i=1}^n (\lambda_i - \mu)= det(D-\mu I + \beta_1 \bv_1\bv_1^T + \beta_2 \bv_2 \bv_2^T) \\
&=det(D-\mu I + \beta_1 \bv_1\bv_1^T )(1+\beta_2 \bv_2^T(D-\mu I + \beta_1 \bv_1\bv_1^T)^{-1} \bv_2)\\
&= det(D-\mu I) (1+ \beta_1 \bv_1^T(D-\mu I)^{-1}\bv_1 )(1+\beta_2 \bv_2^T(D-\mu I + \beta_1 \bv_1\bv_1^T)^{-1} \bv_2)\\
&= det(D-\mu I) (1+ \beta_1 \bv_1^T(D-\mu I)^{-1}\bv_1 )\left(1+\beta_2 \bv_2^T\left\{(D-\mu I)^{-1}
- \frac{\beta_1 (D-\mu I)^{-1} \bv_1 \bv_1^T (D-\mu I)^{-1}}{1+\beta_1 \bv_1^T (D-\mu I)^{-1} \bv_1}\right\} \bv_2\right)\\
&=det(D-\mu I)\left(1+\displaystyle\sum_{q=1}^n \frac{\beta_1v_{1q}^2+\beta_2v_{2q}^2}{d_q-\mu }+\beta_1\beta_2\displaystyle\sum_{q=1}^n \frac{v_{1q}^2}{d_q-\mu}\displaystyle\sum_{q=1}^n \frac{v_{2q}^2}{d_q-\mu}-\beta_1\beta_2\left(\displaystyle\sum_{q=1}^n \frac{v_{1q}v_{2q}}{d_q-\mu}\right)^2\right)\\
&=\displaystyle\prod_{i=1}^n (d_i - \mu) \left(1 +\sum_{q=1}^n \frac{\beta_1v_{1q}^2+\beta_2v_{2q}^2}{d_q-\mu}+\beta_1\beta_2\displaystyle\sum_{q=1}^n \sum_{r= q+1}^n\frac{(v_{1q}v_{2r}-v_{1r}v_{2q})^2}{( d_q-\mu)(d_r-\mu)}\right)\\
&=\displaystyle\prod_{i=1}^n (d_i - \mu)  +\sum_{q=1}^n  \displaystyle\prod_{\substack{i=1 \\i\ne q}}^n (d_i - \mu) (\beta_1v_{1q}^2+\beta_2v_{2q}^2)  +\beta_1\beta_2\displaystyle\sum_{q=1}^n \sum_{r= q+1}^n    \displaystyle\prod_{\substack{i=1 \\i\ne q,r}}^n (d_i - \mu)(v_{1q}v_{2r}-v_{1r}v_{2q})^2 .
\end{array}
$$
Setting $\mu=d_k$, $k=1,\cdots,n$, we obtain relatively simple equations. However, unlike the rank one modification, 
it is obvious that we cannot solve these equations for $\hat{\bv_1}$ and $\hat{\bv_2}$. We tried some ad-hoc method to compute approximations of $\hat{\bv_1}$ and $\hat{\bv_2}$. The advantage of this method is that it requires only $O(n^2)$ operations. Since this approximation leads the orthogonality problem, the results were not so satisfactory for $n\ge 80$.

\subsection{Method 2 - Repeated applications of  Gu and Eisenstat}
 We apply the methods used in Gu and Eisenstat\cite{Gu} by regarding our decomposition as a repeated rank one modifications.
\begin{enumerate}
\item Apply the Gu's method to $B_1$ of (\ref{Bi}) (resp. $B_2$) then
\item apply the Gu's method to the first expression of $A$ in (\ref{A}) (resp. second expression of $A$).
\end{enumerate}
Although this method needs $O(n^3)$ operations, the orthogonality problem does not arise.
Hence, we can apply this algorithm in a recursive way.

\section{Numerical results}
In this section we compare the execution time and accuracy measures between CDC and our rank two modification divide and conquer(RTDC).
The algorithm was run on a 32 bit laptop computer which has machine epsilon $\eps=2.2204\times 10^{-16}$.
As a test matrix, we have chosen a typical matrix arising from solving Laplace equations on squares using the finite difference method.
The test matrix is
$$
A=\frac1{h^2}
\begin{bmatrix}
 B&-I&&\b0\\
-I&\ddots&\ddots&\\
&\ddots& B&-I\\
\b0&&-I&B
\end{bmatrix},\,\, \mbox{where } B= \begin{bmatrix} 4&-1&&\b0\\
 -1&\ddots&\ddots&\\
& \ddots&4&-1\\
\b0&&-1&4
\end{bmatrix}
$$
where $h=\frac{1}{\sqrt n+1}$and the eigenvalues and eigenvectors of $A$ are well-known. We used Householder's
method to tridiagonalize it.
We use the same residual and orthogonality measures defined by \cite{Gu}
$$
\mathcal{R}=\frac{||A\Hat Q-\hat Q \hat \Lambda ||_2}{n\eps||A||_2} \quad \text{and} \quad \mathcal{O}=\frac{||I-\hat Q^T \hat Q||_2}{n\eps},
$$ where $\hat Q \hat \Lambda \hat Q^T$ is the computed spectral decomposition of $A$.

\begin{table}[h]
\centering
\begin{tabular}{|c|c|c|c|c|}
\hline
\multirow{2}{*}{n} & \multirow{2}{*}{\begin{tabular}[c]{@{}c@{}}Accuracy\\ measures\end{tabular}} & \multicolumn{3}{c|}{Eigensolver} \\ \cline{3-5}
                    &                      & QR & CDC & RTDC \\ \hline
\multirow{2}{*}{9} & $\mathcal{R}$ & 0.212 & 0.123 & 0.226 \\ \cline{2-5}
                    & $\mathcal{O}$ & 0.391 & 0.334 & 0.156 \\ \hline
\multirow{2}{*}{25}  & $\mathcal{R}$   & 0.202 & 0.226 & 0.224 \\ \cline{2-5}
                    & $\mathcal{O}$   & 0.292 & 0.411 & 0.174 \\ \hline
\multirow{2}{*}{100}  & $\mathcal{R}$   & 0.260 & 0.199 & 0.190 \\ \cline{2-5}
		   & $\mathcal{O}$  & 0.269 & 0.163 & 0.113 \\ \hline
\multirow{2}{*}{400}  & $\mathcal{R}$   & 0.235 & 0.177 & 0.318 \\ \cline{2-5}
		   & $\mathcal{O}$  & 0.081 & 0.075 & 0.068 \\ \hline
\end{tabular}
\caption{Accuracy measures of each algorithms}
\end{table}
We now compare computational complexity.
Since we have to use QR (or similar) method to solve the subproblems, we need $O(n^3)$ operations.
In this experiment, we used implicit QR for all algorithms.
Since we can divide $A$ into the three roughly equal sized sub-matrices, RTDC takes $\frac{4}{9}$ flops of CDC to calculate the eigenvalues.
For eigenvectors, CDC costs $\frac{1}{2}n^3+O(n^2)$ arithmetic operations. However, RTDC takes $\frac{1}{3}n^3+O(n^2)$ operations when we do not fix the orthogonality problem or use the first method in Section \ref{fixing}.
If we fix the orthogonality problem by the second method, we need extra $\frac{4}{9}n^3$ operations.
However, since eigenvector calculations consist of nothing else but matrix multiplications in (\ref{rank-two-modified}), it can be effectively parallelized.
Since our algorithm can calculate the eigenvalue twice as fast as the original CDC, our algorithm
 is easy to  parallelize.

\section{Conclusion}
We have extended the work of Cuppen's divide and conquer algorithm to rank two modification. 
Unlike the CDC, the deflation steps in our algorithm are nontrivial and the number of the eigenvalues vary 
in each interval $I_i$. We have shown how to classify and deflate the case when eigenvalues of $T_i$ coincide with those of $A$.
Also, we have found an algorithm to find the eigenvalue distribution.
In the original CDC algorithm, each subinterval contains exactly one eigenvalue, but in our algorithm such relation no longer holds.
Using the secular equation we have shown that each interval has either no eigenvalues, one or two eigenvalues.
Eigenvectors can also be computed by the  algorithm. However, to improve the orthogonality problem, we suggested to apply the idea used in the rank one modification. Since the orthogonality problem does not arise in this case, we can apply this algorithm in a recursive way. We believe that our algorithm can be effectively parallelized, and we will leave this as a future work.

\end{document}